\documentclass{article}

\usepackage{amsfonts, amsmath, amssymb, amsthm, amscd}
\usepackage{verbatim}
\usepackage[dvips]{graphicx}
\usepackage{graphics}
\usepackage[all,2cell]{xypic}
\objectmargin+{1mm}
\labelmargin+{0.8mm}
\SelectTips{cm}{12}

\theoremstyle{plain}  
\newtheorem{thm}{Theorem}[section]

\newtheorem{cor}[thm]{Corollary}
\newtheorem{lem}[thm]{Lemma}
\newtheorem{prop}[thm]{Proposition}

\theoremstyle{definition}

\newtheorem{df}[thm]{Definition}
\newtheorem{ex}[thm]{Example}

\newtheorem{nt}[thm]{Notations}

\newtheorem{rem}[thm]{Remark}

\newtheorem{para}[thm]{}

\theoremstyle{remark}
\newtheorem*{claim}{Claim}

\DeclareMathOperator{\id}{id}
\DeclareMathOperator{\isoto}{\overset{\scriptstyle{\sim}}{\to}}
\DeclareMathOperator{\rinf}{\rightarrowtail}

\DeclareMathOperator{\rdef}{\twoheadrightarrow}

\DeclareMathOperator{\rinc}{\hookrightarrow}

\newcommand{\onto}[1]{\stackrel{#1}{\to}}

\DeclareMathOperator{\hdim}{hdim}
\DeclareMathOperator{\Ext}{Ext}

\DeclareMathOperator{\im}{Im}
\DeclareMathOperator{\coker}{Coker}

\DeclareMathOperator{\kos}{Kos}
\DeclareMathOperator{\Ch}{\mathbf{Ch}}

\DeclareMathOperator{\op}{op}

\DeclareMathOperator{\Hom}{Hom}
\DeclareMathOperator{\HOM}{\mathcal{HOM}}

\DeclareMathOperator{\colim}{colim}

\DeclareMathOperator{\Homo}{H}

\DeclareMathOperator{\bbN}{\mathbb{N}}

\DeclareMathOperator{\bbZ}{\mathbb{Z}}

\DeclareMathOperator{\fe}{\mathfrak{e}}

\DeclareMathOperator{\ii}{\mathfrak{i}}
\DeclareMathOperator{\jj}{\mathfrak{j}}

\DeclareMathOperator{\cA}{\mathcal{A}}
\DeclareMathOperator{\cB}{\mathcal{B}}
\DeclareMathOperator{\cC}{\mathcal{C}}

\DeclareMathOperator{\cE}{\mathcal{E}}
\DeclareMathOperator{\cF}{\mathcal{F}}

\DeclareMathOperator{\cI}{\mathcal{I}}

\DeclareMathOperator{\cM}{\mathcal{M}}

\DeclareMathOperator{\cP}{\mathcal{P}}

\DeclareMathOperator{\cX}{\mathcal{X}}
\DeclareMathOperator{\cY}{\mathcal{Y}}

\newcommand{\Ab}{\operatorname{\bf Ab}}

\DeclareMathOperator{\Cat}{\mathbf{Cat}}
\newcommand{\Mod}{\operatorname{\bf Mod}}

\DeclareMathOperator{\Lex}{\mathbf{Lex}}
\DeclareMathOperator{\Kos}{\mathbf{Kos}}

\newcommand{\End}{\operatorname{\bf End}}

\newcommand{\gr}{\operatorname{gr}}
\newcommand{\nil}{\operatorname{nil}}
\newcommand{\tf}{\operatorname{tf}}

\UseAllTwocells

\def\sn{\smallskip\noindent}

\def\enumidef{\renewcommand{\labelenumi}{$\mathrm{(\arabic{enumi})}$}}

\title{Homotopy invariance of higher $K$-theory for 
abelian categories}
\author{Satoshi Mochizuki and Akiyoshi Sannai}
\date{}

\begin{document}
\maketitle

\begin{abstract}
The main theorem in this paper is that 
the base change functor from an abelian category $\cA$ to 
its polynomial category in the sense of Schlichting
$-\otimes_{\cA}\bbZ[t]:\cA \to \cA[t]$ 
induces an isomorphism on their $K$-theories 
if $\cA$ is noetherian and has enough projective objects. 
The main theorem implies the well-known fact 
that $\mathbb{A}^1$-homotopy invariance of $K'$-theory for noetherian schemes. 
\end{abstract}

\section{Introduction}

Contrary to the importance of $\mathbb{A}^1$-homotopy invariance 
in the motivic homotopy theory \cite{Voe98}, \cite{MV99} and \cite{Voe00}, 
the homotopy invariance of $K'$-theory for noetherian schemes still 
has been mysterious in the following sense. 
After \cite{Sch11}, 
every fundamental theorems except for 
homotopy invariance of $K'$-theory, 
the d\'evissage theorem in \cite{Qui73} and 
the cell filtration theorem in \cite{Wal85} 
are corollaries of 
Thomason-Schlichting localization theorem and 
in the view of non-commutative motive theory \cite{CT09} 
or motive theory for $\infty$-categories 
\cite{BGT10}, 
non-connective $K$-theory is the universal localizing invariant. 
To relate motivic homotopy theory with motive theory 
for DG or $\infty$-categories, 
it is important to make clear the homotopy invariance of $K'$-theory in 
the view of motive theory for higher categories. 
Many authors have already defined affine line over certain categories as in 
\cite{GM96}, 
\cite{Sch06} and \cite{BGT10}. 
The main objective in the paper is 
to examine the homotopy invariance of 
$K$-theory for abelian categories by taking Schlichting polynomial categories. 
Let us recall the definition of polynomial categories.

\begin{df}
\label{df:introdf}

\begin{enumerate}
\enumidef
\item 
(Notation~\ref{nt:endcat}) 
For a category $\cC$, 
let us denote 
the {\bf category of endomorphisms} 
in $\cC$ by $\End \cC$. 
Namely, an object in $\End\cC$ 
is a pair $(x,\phi)$ 
consisting of an object $x$ in $\cC$ 
and morphism $\phi:x\to x$ and a morphism 
between $(x,\phi) \to (y,\psi)$ 
is a morphism 
$f:x \to y$ in $\cC$ such that 
$\psi f= f\phi$. 

\item 
From now on, let $\cA$ be an abelian category. 
Let us denote the category of left exact functors from 
$\cA^{\op}$ to the category of abelian groups $\Ab$ by 
$\Lex \cA$. 
The category $\Lex \cA$ is a Grothendieck abelian category and 
the Yoneda embedding $y:\cA \to \Lex \cA$ is exact and reflects 
exactness.

\item 
(Notation~\ref{nt:noetherianobj}) 
We say an object $x$ in $\cA$ is 
{\bf noetherian} 
if every ascending filtration of subobjects of $x$ is stational. 
We say $\cA$ is {\bf noetherian} if every object in $\cA$ is noetherian. 

\item 
(Definition~\ref{df:S-poly cat def}) 
Let us assume that $\cA$ is a noetherian abelian category and 
let us denote the full subcategory of 
noetherian objects 
in $\End \Lex\cA$ 
by $\cA[t]$ and 
call it the {\bf noetherian polynomial category} 
over $\cA$. 
We can prove that $\cA[t]$ is an abelian category. 
(See Lemma~\ref{lem:noetherian}).

\item 
For an object $a$ in $\cA$, 
let us define an object $a[t](=(a[t],t))$ in $\End \Lex\cA$ as follows. 
The underlying object $a[t]$ is 
$\displaystyle{\bigoplus_{n=0}^{\infty}}at^i$ 
where $at^i$ is a copy of $a$. 
The endomorphism $t:a[t] \to a[t]$ is defined by 
the identity morphisms $at^i \to at^{i+1}$ in each components. 
We can prove that if $a$ is noetherian in $\cA$, then 
$a[t]$ is noetherian in $\cA[t]$. (See Theorem~\ref{thm:Abst Hilb basis}). 
We call the association $-\otimes_{\cA}\bbZ[t]:\cA \to \cA[t]$, 
$a \mapsto a[t]$ the {\bf base change functor} which is an exact functor. 
\end{enumerate}
\end{df}

\sn
The main theorem is the following. 

\begin{thm}[Theorem~\ref{thm:main thm}]
\label{thm:introthm} 
Let $\cA$ be a noetherian abelian category 
which has enough projective objects. 
The functor $-\otimes_{\cA} \bbZ[t]:\cA \to \cA[t]$ induces the 
isomorphism on their $K$-theories
$$K(\cA)\isoto K(\cA[t]).$$
\end{thm}

\sn
\textbf{Conventions.} 
In this note, 
basically we follow the notation of 
exact categories for \cite{Kel90} and 
algebraic $K$-theory for 
\cite{Qui73} and \cite{Wal85}. 
For example, we call admissible monomorphisms 
(resp. 
admissible epimorphism and admissible short exact sequences) 
inflations (resp. deflations, conflations). 
We also call a category with cofibrations and weak equivalences 
a Waldhausen category. 
Let us denote the set of all natural numbers by $\bbN$. 
We consider it to be totally ordered set 
with the usual order. 
For a Waldhausen category, we write the specific zero object 
by the same letter $\ast$. 
Let us denote the $2$-category of essentially small categories by $\Cat$. 
For categories $\cX$, $\cY$, 
we denote the (large) category of 
functors from $\cX$ to $\cY$ by 
$\HOM(\cX,\cY)$ 
For any (left) noetherian ring $A$, 
we denote the category of finitely generated $A$-modules by $\cM_A$. 
Throughout the paper, 
we use the letter $\cA$ to denote an essentially small abelian category. 
For an object $x$ in $\cA$ and a finite family 
$\{x_i\}_{1\leq i\leq m}$ 
of subobjects of $x$, $\sum_{i=1}^m x_i$ means 
the minimal sub object of $x$ which contains all $x_i$. 
For an additive category $\cB$, we write the category of chain complexes 
on $\cB$ by $\Ch(\cB)$.

\sn
{\bf Acknowledgements.} 
The authors wish to express their deep gratitude 
to Marco Schlichting 
for instructing them in the proof of 
abstract Hilbert basis theorem \ref{thm:Abst Hilb basis}.

\section{Polynomial categories}
\label{sec:polycat}

In this section, let us recall the notation of 
polynomial abelian categories from \cite{Sch00} or \cite{Sch06}.

\subsection{End categories}

\begin{nt}
\label{nt:endcat}
For a category $\cC$, 
let us denote 
the {\bf category of endomorphisms} 
in $\cC$ by $\End \cC$. 
Namely, an object in $\End\cC$ 
is a pair $(x,\phi)$ 
consisting of an object $x$ in $\cC$ 
and morphism $\phi:x\to x$ and a morphism 
between $(x,\phi) \to (y,\psi)$ 
is a morphism 
$f:x \to y$ in $\cC$ such that 
$\psi f= f\phi$. 
For any functor $F:\cC \to \cC'$, 
we have the functor
$\End F:\End \cC \to \End \cC'$, 
$(x,\phi)\mapsto (Fx,F\phi)$. 
Moreover for any 
natural transformation $\theta:F\to F'$ 
between functors $F$, $F:\cC \to \cC'$, 
we have a natural transformation 
$\End \theta:\End F \to \End F'$ 
defined by the formula 
$\End \theta (x,\phi):=\theta(x)$ for any object $(x,\phi)$. 
This association gives a $2$-functor
$$\End:\Cat \to \Cat.$$
We have natural transformations 
$i:\id_{\Cat} \to \End$ and 
$U:\End \to \id_{\Cat}$ defined by 
$i(\cC):\cC \to \End \cC$, 
$x\mapsto (x,\id_x)$ and 
$U(\cC):\End \cC \to \cC$, 
$(x,\phi) \mapsto x$ for each category $\cC$.  
\end{nt}

\begin{rem}
\label{rem:limitinendcat}
Let $\cC$ be a category and 
$F:\cI \to \End \cC$, 
$i\mapsto (x_i,\phi_i)$ be a functor. 
Let us assume that 
there is a limit $\lim x_i$ 
(resp. colimit $\colim x_i$) in $\cC$. 
Then we have $\lim F_i=(\lim x_i,\lim \phi_i)$ 
(resp. $\colim F_i=(\colim x_i,\colim \phi_i)$.) 
In particular, 
if $\cC$ is additive (resp. abelian), 
then $\End \cC$ is also additive (resp. abelian). 
Moreover if $\cC$ is an exact category 
(resp. a category with cofibration), 
then $\End \cC$ naturally becomes an exact category 
(resp. a category with cofibration.) 
Here a sequence 
$(x,\phi) \to (y,\psi) \to (z,\xi)$ is a conflation
if and only if $x \to y \to z$ is a conflation in $\cC$. 
(resp. a morphism $(x,\phi) \onto{u} (y,\psi)$ is a cofibration 
if and only if 
$u:x\to y$ is a cofibration in $\cC$.) 
Moreover if $w$ is 
a class of morphisms in $\cC$ 
which satisfies the axioms of 
Waldhausen categories 
(and its dual), 
then the class of all morphisms 
in $\End \cC$ 
which is in $w$ also satisfies 
the axioms of 
Waldhausen categories (and its dual). 
\end{rem}

\begin{rem}
\label{rem:GMnotation}
In \cite[III. 5.15]{GM96}, 
for a category $\cC$, 
the category $\End \cC$ is called 
the {\bf polynomial category} over $\cC$ 
and denoted by $\cC[T]$. 
For let $A$ be 
a commutative ring with unit and 
$\Mod(A)$ the category of 
$A$-Modules, 
then we have the canonical category isomorphism 
$$\Mod(A[T]) \isoto (\Mod(A))[T],\ \ M \mapsto (M,T)$$
where $A[T]$ is 
the polynomial ring over $A$ 
and $T$ means the endomorphism 
$T:M \to M$, $x \mapsto Tx$. 
Moreover in general 
for any abelian category $\cA$, 
we have the equality 
$$\hdim \cA[T]=\hdim \cA +1$$
where $\hdim \cA$ is the 
{\bf homological dimension} 
of $\cA$ which is defined by 
$$\hdim \cA:=\max\{n; \Ext^{n}(x,y)\neq 0
 \ \ \ \text{for any objects $x$, $y$}\}.$$
But obviously for any left noetherian ring $A$, 
$(\cM_A)[T]$ and $\cM_{A[T]}$ are different categories. 
The main reasons is that 
$A[T]$ is not finitely generated as an $A$-module. 
In particular, 
the object $(A[T],T)$ is in $(\Mod(A))[T]$ 
but not in $(\cM_A)[T]$. 
In the subsection~\ref{subsec:Sch poly cat}, 
we define the noetherian polynomial categories 
over noetherian abelian category which is introduced by 
Schlichting in \cite{Sch06}. 
In the notion, we have the canonical category equivalence 
between $\cM_{A[t]}$ and $(\cM_A)[t]$. 
See \ref{ex:S-poly}.
\end{rem}

\subsection{Noetherian objects}
\label{subsec:Noeobj}

In this subsection, 
we develop 
the theory of noetherian objects 
in an exact categories 
which is slightly different from 
the usual notation in the category theory.

\begin{nt}
\label{nt:noetherianobj}
Let $\cE$ be an exact category 
and $x$ an object in $\cE$. 
We say $x$ is a {\bf noetherian object} if 
its every ascending filtration of admissible subobjects 
$$x_0\rinf x_1 \rinf x_2 \rinf \cdots$$
is stational. 
We say $\cE$ is a {\bf noetherian} category if every object 
in $\cE$ is noetherian. 
\end{nt}

\sn
We can easily prove the following lemmas.

\begin{lem}
\label{lem:noetherian}
Let $\cE$ be an exact category. 
Then we have the following. 
\begin{enumerate}
\enumidef
\item 
Let $x \rinf y \rdef z$ 
be a conflation in $\cE$. 
If $y$ is noetherian, then $x$ and $z$ are also.

\item 
For noetherian objects $x$, $y$ in $\cE$, 
$x\oplus y$ is also noetherian. 

\item
Moreover assume that $\cE$ is abelian, 
then the converse of 
$\mathrm{(1)}$ is true. 
Namely, in the notation $\mathrm{(1)}$, 
if $x$ and $z$ are noetherian, then $y$ is also. 
\end{enumerate}
\end{lem}

\begin{lem}
\label{lem:faithfulexact}
For any exact faithful functor $F:\cA \to \cB$ between 
abelian categories and an object $x$ in $\cA$, 
if $Fx$ is noetherian, then $x$ is also noetherian. 
\end{lem}

\subsection{Schlichting polynomial category}
\label{subsec:Sch poly cat}

\begin{para}
\label{para:GQ-emmbeding}
For an exact category $\cE$, 
let us denote the category of left exact functors from 
$\cE^{\op}$ to the category of abelian groups $\Ab$ by 
$\Lex \cE$. 
The category $\Lex \cE$ is a Grothendieck abelian category and 
the Yoneda embedding $y:\cE \to \Lex \cE$ is exact and reflects 
exactness.
\end{para}

\begin{para}
\label{para:polynomial category}
For an object $a$ in an additive category $\cA$, 
let us define an object $a[t](=(a[t],t))$ in $\End \Lex\cA$ as follows. 
The underlying object $a[t]$ is 
$\displaystyle{\bigoplus_{n=0}^{\infty}}at^i$ 
where $at^i$ is a copy of $a$. 
The endomorphism $t:a[t] \to a[t]$ is defined by 
the identity morphisms $at^i \to at^{i+1}$ in each components. 
\end{para}

\sn
The following theorem is proved in \cite[9.10 b]{Sch00}

\begin{thm}[\bf Abstract Hilbert basis theorem]
\label{thm:Abst Hilb basis}
For any noetherian object $a$ in an abelian category $\cA$, 
$a[t]$ is also a noetherian object in $\End\Lex \cA$.
\end{thm}

\begin{df}[\bf Schlichting polynomial category]
\label{df:S-poly cat def} 
Let us assume that $\cA$ is a noetherian abelian category and 
let us denote the full subcategory of 
noetherian objects 
in $\End \Lex\cA$ 
by $\cA[t]$ and 
call it the {\bf noetherian polynomial category} 
over $\cA$. 
By the virtue of \ref{lem:noetherian} and \ref{thm:Abst Hilb basis}, 
we learn that $\cA[t]$ is a noetherian abelian category. 
\end{df}

\begin{rem}
\label{rem:other df of S-poly cat}
We can easily prove that an object $x$ in $\End \Lex\cA$ is 
in $\cA[t]$ if and only if 
there is a deflation $a[t] \rdef x$ for some object 
$a$ in $\cA$.
\end{rem}

\begin{ex}
\label{ex:morphismofa[t]}
For any noetherian objects $a$, $b$ in $\cA$ and 
a morphism $f:a[t] \to b[t]$ in $\cA[t]$, 
there is a positive integer $m$ such that $f(a)$ is in 
$\displaystyle{\bigoplus_{i=1}^m bt^i}$. 
Since the morphism $f$ is recovered by the restriction 
$a \to a[t] \onto{f} b[t]$, 
$f$ is determined by 
morphisms $c_i:a \to b$ ($0\leq i\leq m$) in $\cA$. 
We write $f$ by $\displaystyle{\sum_{i=1}^mc_it^i}$.
\end{ex}

\begin{ex}
\label{ex:S-poly}
We have the category equivalence
$$\cM_{A[t]}\isoto (\cM_A)[t],\ \ M \mapsto (M,t).$$
\end{ex}

\begin{lem}
\label{lem:projectivity}
For any projective object $a$ in a noetherian abelian category $\cA$, 
$a[t]$ is a projective object in $\cA[t]$.
\end{lem}

\begin{proof}[\bf Proof]
let us consider the diagram in $\cA[t]$ below.
$$\footnotesize{\xymatrix{
a[t] \ar[rd]^f \ar@{-->}[d]_h \\
x \ar@{->>}[r]_{g} & y & .
}}$$ 
Since there is 
an epimorphism $a \rdef f(a)$ from 
a noetherian object in $\Lex \cA$, 
the object $f(a)$ in $\Lex \cA$ is noetherian. 
Therefore $f(a)$ is in $\cA$ by \cite[5.8.8, 5.8.9]{Pop73}. 
Since $x$ is an object in $\cA[t]$, 
there are an object $b$ in $\cA$ and 
a epimorphism $s:b[t] \rdef x$ in $\End\Lex \cA$. 
Then since $f(a)$ is a noetherian object in $\cA$, 
there is a positive integer $m$ such that 
$\displaystyle{\im(\bigoplus_{i=1}^mbt^i \to x \to y)}$ contains $f(a)$. 
We put $\displaystyle{u=s(\bigoplus_{i=1}^mbt^i)}$ and $v=g(u)$. 
Then $u$ and $v$ are noetherian objects in $\Lex\cA$ 
and therefore 
they are in $\cA$ by Ibid. 
We have the diagram in $\cA$ below. 
$$\footnotesize{\xymatrix{
a \ar[rd]^{f'} \ar@{-->}[d]_{h'} \\
u \ar@{->>}[r]_{g'} & v & .
}}$$ 
Therefore by projectivity of $a$, we have the dotted morphism $h':a \to u$ 
which makes the diagram above commutative. 
The composition $a \onto{h'} u \to x$ induces the desired morphism 
$h:a[t] \to x$ which makes the first diagram above commutative.
\end{proof}

\section{Graded categories}
\label{sec:graded cat}

In this section, we will introduce the notion of 
(noetherian) graded categories 
over categories and calculate the $K$-theory of 
noetherian graded categories 
over noetherian abelian categories.

\subsection{Fundamental properties of graded categories}
\label{subsec:grad cat}

\begin{para}
\label{para:cat<n>}
For a positive integer $n$, we define the category 
$<n>$ as follows. 
The class of objects of $<n>$ is just the set of all natural numbers 
$\bbN$. 
The class of morphisms of $<n>$ is generated by 
morphisms $\psi^{i}_m:m \to m+1$ 
for any $m$ in $\bbN$ and $1\leq i\leq n$ 
which subject to the equalities 
$\psi^i_{m+1}\psi^j_m=\psi^j_{m+1}\psi^i_m$ 
for each $m$ in $\bbN$ and $1\leq i,\ j \leq n$. 
\end{para}

\begin{df}[\bf Graded categories]
\label{df:graded cat}
For a positive integer $n$ and a category $\cC$, 
we put $\cC_{\gr}[n]:=\HOM(<n>,\cC)$ and call it 
the {\bf category of ($n$-)graded category over $\cC$}. 
For an object $x$, a morphism $f:x \to y$ in $\cC_{\gr}[n]$, 
we write $x(m)$, $x(\psi_m^i)$ and $f(m)$ 
by $x_m$, $\psi_m^{i,x}$ or shortly 
$\psi_m^i$ and $f_m$ respectively. 
\end{df}

\begin{rem}
\label{rem:limit in graded cat}
We can calculate a (co)limit in $\cC_{\gr}[n]$ by 
term-wise (co)limit in $\cC$. 
In particular, if 
$\cC$ is additive (resp. abelian) 
then $\cC_{\gr}[n]$ is also additive (resp. abelian). 
Moreover if 
$\cC$ is a category with cofibration 
(resp. an exact category), 
then $\cC_{\gr}[n]$ naturally becomes 
a category with cofibration 
(resp. an exact category). 
Here a sequence 
$x \to y \to z$ is a conflation 
(resp. a morphism $x \to y$ is a cofibration) 
if it is term-wisely in $\cC$. 
Moreover if $w$ is a class of morphisms in $\cC$ 
which satisfies the axioms of Waldhausen categories 
(and its dual), then 
 the class of all morphisms $lw$ in $\cC_{\gr}[n]$ 
those of morphisms $f$ such that $f_m$ is in $w$ for 
all natural number $m$ also satisfies the axioms 
of Waldhausen categories (and its dual).
\end{rem}

\begin{para}
\label{para:cC'df}
For an exact category $\cE$ and a positive integer $n$, 
we denote the full subcategory of all noetherian objects 
in $\cE_{\gr}[n]$ by $\cE_{\gr}'[n]$. 
In particular if $\cE$ is an abelian category then 
$\cE'_{\gr}[n]$ is a noetherian abelian category by \ref{lem:noetherian}. 
In the case, we call $\cE'_{\gr}[n]$ the 
{\bf noetherian ($n$-)graded category over $\cE$}. 
\end{para}

\begin{ex}
\label{ex:gradedcategory}
For a noetherian commutative ring with unit $A$ and 
$\cE=\cM_A$, 
$\cE'_{\gr}[n]$ is just the category of 
finitely generated graded $A[t_1,\cdots,t_n]$-modules.
\end{ex}

\begin{proof}[\bf Proof]
Let $\cF$ be a category of 
finitely generated graded $A[t_1,\cdots,t_n]$-modules. 
Any object $x$ in $\cF$ is considered to be an object in $\cE_{\gr}[n]$ 
in the following way. 
Let us define the functor $x':<n> \to \cE$ 
by $k \mapsto x_k$ and $(\psi^i:k\to k+1)\mapsto (t_i:x_k \to x_{k+1})$. 
The association $x \mapsto x'$ induces 
a category equivalence $\cF \isoto \cE'_{\gr}[n]$.
\end{proof}

\begin{nt}[\bf Degree shift]
\label{nt:degreeshiftofobjects}
Let $\cC$ be a pointed category and $k$ an integer. 
We define the functor 
$(k):\cC_{\gr}[n] \to \cC_{\gr}[n]$, $x \mapsto x(k)$. 
For any object $x$ and 
any morphism $f:x \to y$ in $\cC_{\gr}[n]$, 
we define the object $x(k)$ and 
the morphism $f(k):x(k) \to y(k)$ as follows. 
We put 
$$
x(k)_m=
\begin{cases}
x_{m+k} & \text{if $m\geq -k$}\\
0 & \text{if $m < -k$}
\end{cases},\ \ 
\psi^{i,x(k)}_m:=
\begin{cases}
\psi^{i,x}_{m+k} & \text{if $m\geq -k$}\\
0 & \text{if $m < -k$}
\end{cases}
\ \ \text{and}\ \ 
f(k)_m:=
\begin{cases}
f_{m+k} & \text{if $m\geq -k$}\\
0 & \text{if $m < -k$}
\end{cases}
.$$
For any object $x$ in $\cC_{\gr}[n]$ and a positive integer $k$, 
we have the canonical morphism $\psi_x^{i,k}(=\psi^i):x(-k) \to x(-k+1)$ 
defined by $\psi_{m-k}^i:{x(-k)}_m=x_{m-k} \to {x(-k+1)}_m=x_{m-k+1}$ 
for each $m$ in $\bbN$. 
\end{nt}

\begin{nt}
\label{nt:compoftranslation}
For any natural numbers $m$ and $k$, 
an object $x$ in $\cC_{\gr}[n]$ 
and a multi index $\ii=(i_1,\cdots,i_n)\in\bbN^n$, 
we define the morphism 
$\psi^{\ii,k}_x(=\psi^{\ii}):x(-(\sum_{j=1}^n i_j+k)) 
\to x(-k)$ by 
$$\psi^{\ii}={(\psi^n)}^{i_n}{(\psi^{n-1})}^{i_{n-1}}\cdots
{(\psi^2)}^{i_2}{(\psi^1)}^{i_1}.$$
\end{nt}

\begin{nt}[\bf Free graded object]
\label{nt:free graded object}
Let $\cC$ be an additive category 
and $n$ a positive integer. 
We define the functor 
$\cF_{\cC}[n](=\cF[n]):\cC \to \cC_{\gr}[n]$ 
in the following way. 
For any object $x$ in $\cC$, 
we define the object 
$\cF[n](x)=x[\{\psi^i\}_{1\leq i\leq m}]$ 
in $\cC_{\gr}[n]$ as follows. 
We put 
$$\cF[n](x)_m:=\underset{
\substack{\ii=(i_1,\cdots,i_n)\in\bbN^n \\ 
\sum^n_{j=1}i_j=m}}{\bigoplus} 
x_{\ii}$$ 
where $x_{\ii}$ is a copy of $x$. 
$x_{\ii}$ ($\displaystyle{\sum_{j=1}^n i_j=m}$) components of the morphisms 
$\psi_m^{k,\cF[n](x)}:\cF[n](x)_m \to \cF[n](x)_{m+1}$  
defined by  
$\id:x_{\ii} \to x_{\ii+\fe_k}$ 
where $\fe_k$ is the $k$-th unit vector. 
\end{nt}

\begin{para}
\label{para:can mor} 
Let $\cC$ be an additive category and $k$ a natural number. 
For any object $x$ in $\cC_{\gr}[n]$, 
we have the canonical morphism 
$\cF[n](x_k)(-k) \to x$ which is defined as follows.
For each $m\geq k$ and $\ii=(i_1,\cdots,i_n)\in\bbN^n$ 
such that $\displaystyle{\sum_{j=1}^n i_j=m-k}$, 
on the $x_{\ii}$ component of $\cF[n](x_k)(-k)_m$, 
the morphism is defined by 
$\psi_m^{\ii}:x_{\ii} \to x_m$. 
\end{para}

\begin{rem}
\label{rem:adjointnessofF[n]}
Let $\cC$ be an additive category. 
Then the functor $\cF[n]:\cC \to \cC_{\gr}[n]$ is 
the left adjoint functor of the functor 
$\cC_{\gr}[n] \to \cC$, $y \mapsto y_0$. 
Namely for any object $x$ in $\cC$ 
and any object $y$ in $\cC_{\gr}[n]$, we have 
the functorial isomorphism
$$\Hom_{\cC}(x,y_0)\isoto\Hom_{\cC_{\gr}[n]}(\cF[n](x),y),\ \ f \mapsto 
(\cF[n](x)\onto{\cF[n](f)}\cF[n](y_0) \to y).$$
\end{rem}

\begin{ex}
\label{ex:mor between freeobj} 
For any objects $x$, $y$ in an additive category $\cC$, 
a positive integer $k$,  
and family of morphisms 
$\{c_{\ii}\}_{\ii=(i_1,\cdot,i_n)\in\bbN^n,\ \sum i_j=k}$ 
from $x$ to $y$, 
we define the morphism $\sum c_{\ii}\psi^{\ii}:\cF[n](x)(-k) \to \cF[n](y)$ 
by $c_{\ii}:x_{\jj} \to x_{\jj+\ii}$ on its $x_{\jj}$ component to 
$x_{\jj+\ii}$ component. 
\end{ex}

\begin{lem}
\label{lem:fundamental facts about}
Let $\cA$ be a noetherian abelian category and 
$n$ a positive integer. 
We have the following assertions.
\begin{enumerate}
\enumidef
\item
For any object $x$ in $\cA$, 
$\cF[n](x)$ is a noetherian object in $\cA_{\gr}[n]$. 
In particular, 
we have the exact functor 
$$\cF_{\cA}[n]:\cA \to \cA'_{\gr}[n].$$
\item 
For any object $x$ in $\cA'_{\gr}[n]$, 
there is a natural number $m$ such that the canonical morphism 
as in \ref{para:can mor} 
$$\displaystyle{\overset{m}{\underset{k=0}{\bigoplus}}
\cF[n](x_k)(-k) \to x}$$ 
is an epimorphism. 
\item 
If $x$ is a projective object in $\cA$, 
then $\cF[n](x)$ is also a projective object in $\cA_{\gr}[n]$. 
\end{enumerate}
\end{lem}

\begin{proof}[\bf Proof]
$\mathrm{(1)}$ 
We define the functor 
$$\Gamma:\cA_{\gr}[n] \to \End^n\Lex\cA,\ x\mapsto 
(\bigoplus x_{\ii},\ \bigoplus\psi_m^1,\cdots,\bigoplus \psi_m^n ) $$
where $\End^n$ means the $n$-times iteration of the functor $\End$. 
Since $\Lex \cA$ is Grothendieck abelian, 
the functor $\bigoplus$ is exact and 
therefore $\Gamma$ is an exact functor. 
Moreover for a morphism $f:x\to y$ in $\cA_{\gr}[n]$, 
the condition $\Gamma(f)=0$ obviously implies 
the condition $f=0$. 
Hence $\Gamma$ is faithful. 
We can easily check that for any object $x$ in $\cA$, 
$\Gamma(\cF[n](x))=x[t_1,\cdots,t_n]$ 
which is a noetherian object in $\End^n\Lex\cA$ by \ref{thm:Abst Hilb basis}. 
Therefore $\cF[n](x)$ is noetherian by \ref{lem:faithfulexact}. 

\sn
$\mathrm{(2)}$ 
We put $\displaystyle{z_l=\im(\bigoplus_{k=0}^l \cF[n](x_k)(-k) \to x)}$. 
Let us consider the ascending chain of subobjects in $x$
$$z_1\rinf z_2\rinf \cdots \rinf x.$$ 
Since $x$ is a noetherian object, 
there is a natural number $m$ such that 
$z_m=z_{m+1}=\cdots$. 
We claim that the canonical morphism 
$$y:=\bigoplus_{k=0}^i\cF[n](x_k)(-k) \to x$$
is an epimorphism. 
If $k\geq m$, $y_k \to x_k$ is obviously an epimorphism. 
If $k>m$, then we have the equalities
$$\im(y_k \to x_k)={(z_m)}_k={(z_k)}_k=x_k.$$ 
Therefore we get the desired result.

\sn
$\mathrm{(3)}$ 
Let us consider the left diagram in $\cA_{\gr}[n]$ below:
$$\footnotesize{\xymatrix{
\cF[n](x) \ar@{-->}[d]_{h} \ar[rd]^f \\ 
y \ar@{->>}[r]_g & z \ar[r] & 0
}\ \ \ \ 
\xymatrix{ 
x \ar@{-->}[d]_{h_0} \ar[rd]^{f_0} \\ 
y_0 \ar@{->>}[r]_{g_0} & z_0 \ar[r] & 0
}}$$
where $g$ is an epimorphism. 
Then we have the right diagram in $\cA$ above. 
By projectivity of $x$, we have the dotted morphism $h_0$ 
which makes the right diagram above commutative. 
Then by \ref{rem:adjointnessofF[n]}, 
we get $h:\cF[n](x) \to y$ 
which makes the left diagram above commutative. 
\end{proof}

\begin{df}[\bf Canonical filtration] 
\label{df:canfilt}
For any object $x$ in $\cA'_{\gr}[n]$, 
we define the canonical filtration $F_{\bullet}x$ as follows. 
$F_{-1}x=0$ 
and for any $m\geq0$,
$$(F_mx)_k=
\begin{cases}
x_k & \text{if $k\leq m$}\\
\displaystyle{\sum_{\substack{\ii=(i_1,\cdots,i_n)
\in\bbN^n\\\Sigma i_j =k-m}} 
\im\psi_{m}^{\ii}} & \text{if $k> m$}
\end{cases}.$$
\end{df}

\begin{rem}
\label{rem:dim of x}
Since every object $x$ in $\cA'_{\gr}[n]$ 
is noetherian, there is the minimal integer $m$ such that 
$F_mx=F_{m+1}x=\cdots$. 
In the case, 
we can easily prove that $F_mx=x$. 
We call $m$ {\bf degree} of $x$ and denote it by $\deg x$. 
\end{rem}

\subsection{Koszul homologies}
\label{subsec:Kos hom}

In this subsection, we define the Koszul homologies of 
objects in $\cA'_{\gr}[n]$ and as its application, 
we study the $K$-theory of $\cA'_{\gr}[n]$. 

\begin{nt}[\bf Koszul complex]
\label{nt:Koszul complex}
Let $\cC$ be an additive category and $n$ a positive integer. 
For any object $x$ in $\cC_{\gr}[n]$, 
we define the {\bf Koszul complex} $\kos(x)$ associated with $x$ 
as follows. 
For each $0\leq k\leq n$, we put 
$\displaystyle{\Kos(x)_k:=\bigoplus_{\substack{\ii=(i_1,\cdots,i_n)\in[1]^n \\ \sum i_j=k}} 
x_{\ii}}$ 
where $[1]$ is the totally ordered set 
$\{0,1\}$ with the natural order and 
$x_{\ii}$ is a copy of $x(-\sum_{j=1}^n i_j)$. 
The boundary morphism $d^{\Kos(x)}_k:\Kos(x)_k \to \Kos(x)_{k-1}$ is 
defined by $\displaystyle{(-1)^{\sum_{t=j+1}^n i_j}\psi_i:x_{\ii} \to x_{\ii-\fe_j}}$ on its $x_{\ii}$ to $x_{\ii-\fe_j}$ component where 
$\fe_j$ is the $j$-th unit vector. 
The association $x \mapsto \Kos(x)$ defines the exact functor 
$$\Kos:\cC_{\gr}[n] \to \Ch(\cC_{\gr}[n]).$$
\end{nt}

\begin{df}[\bf Koszul homologies]
\label{df:Koszul homologies}
Let $\cE$ be an idempotent complete exact category and $n$ a positive integer. 
We put $\cB:=\Lex \cE$. 
We define the family of functors $\{T_i:\cE_{\gr}[n] \to \cB_{\gr}[n]\}$ by 
$T_i(x):=\Homo_i(\Kos(x))$ for each $x$. 
$T_i(x)$ is said to be the {\bf $i$-th Koszul homology} of $x$. 
Let us notice that for any conflation $x \rinf y \rdef z$ in 
$\cE_{\gr}[n]$, we have a long exact sequences 
$$\cdots \to T_{i+1}(z) \to T_i(x) \to T_i(y) \to T_i(z) 
\to T_{i-1}(x) \to \cdots.$$
\end{df}

\begin{df}[\bf Torsion free objects]
\label{df:Torson free object}
An object $x$ in $\cA'_{\gr}[n]$ is {\bf torsion free} if 
$T_i(x)=0$ for any $i>0$. 
For each non-negative integer $m$, 
we denote the category of torsion free objects 
(of degree less than $m$) 
in $\cA'_{\gr}[n]$ by 
$\cA'_{\gr,\tf}[n]$ (resp. $\cA'_{\gr,\tf,m}[n]$). 
Since $\cA'_{\gr,\tf}[n]$, $\cA'_{\gr,\tf,m}[n]$ are 
closed under extensions in $\cA'_{\gr}[n]$, 
they become exact categories in the natural way. 
\end{df}

\begin{prop}
\label{prop:fun pro Kos hom}
For any objects $x$ in $\cA'_{\gr}[n]$ and $y$ in $\cA$, 
we have the following assertions.
\begin{enumerate}
\enumidef
\item 
For any natural number $k$, 
$\cF[n](y)(-k)$ is torsion free. 

\item 
For each positive integer $s$, 
the assertion 
$T_0(x)_k=0$ for any $k \leq s$ implies $x_k=0$ for any $k \leq s$. 

\item 
We have the equality 
$$T_0(F_px)_k=
\begin{cases}
0 & \text{if $k>p$}\\
T_0(x)_k & \text{if $k\leq p$}
\end{cases}
.$$

\item 
For each natural number $p$, 
there is a canonical epimorphism 
$$\alpha^p:\cF[n](T_0(x)_p)(-p) \rdef F_px/F_{p-1}x.$$

\item 
For each natural number $p$, 
$T_0(\alpha^p)$ is an isomorphism.

\item 
If $T_1(x)$ is trivial, then $\alpha^p$ is an isomorphism.
\end{enumerate}
\end{prop}

\begin{proof}[\bf Proof]
$\mathrm{(1)}$ Since the degree shift functor is exact, 
we have the equality $T_i(x(-k))=T_i(x)(-k)$ 
for any natural numbers 
$i$, $k$. 
Therefore we shall just check that 
$\cF[n](y)$ is torsion free. 
If $\cA$ is the category of finitely generated free $\bbZ$-modules 
$\cP_{\bbZ}$ and $y=\bbZ$, 
then $\cF[n](y)$ is just the $n$-th polynomial 
ring over $\bbZ$, $\bbZ[t_1,\cdots,t_n]$ 
and $T_i(\cF[n](y))$ is the $i$-th homology group of 
the Koszul complex associated with the regular sequence 
$t_1,\cdots,t_n$. 
In the case, it is well-known that 
$T_i(\cF[n](y))=0$ for $i>0$. 
For general $\cA$ and $y$, 
there is the exact functor 
$\cP_{\bbZ} \to \cA$, $\bbZ \mapsto y$ 
which induces $\Ch({(\cP_{\bbZ})}'_{\gr}[n]) \to \Ch(\cA'_{\gr}[n])$ 
and $\Kos(\cF[n](\bbZ))$ goes to $\Kos(\cF[n](y))$ by 
the exact functor. 
Hence we learn that $T_i(\cF[n](y))=0$ for $i>0$. 

\sn
$\mathrm{(2)}$ 
First notice that we have the equalities 
$${T_0(x)}_k=
\begin{cases}
x_0 & \text{if $k=0$}\\
x_k/\im(\psi^1,\cdots,\psi^n) & \text{if $k>0$}
\end{cases}
.
$$
Therefore if 
${T_0(x)}_k=0$ for $k\leq s$, then we have 
$x_0=0$ and $x_k=\im(\psi^1,\cdots,\psi^n)$ for $k\leq s$. 
Hence inductively we notice that $x_k=0$ for $k\leq s$.

\sn
The assertion 
$\mathrm{(3)}$ follows from direct calculation.

\sn
$\mathrm{(4)}$ 
We have the equality 
$$
{(F_px/F_{p-1}x)}_k \isoto
\begin{cases}
0 & \text{if $k < p $}\\
x_p/\im(\psi^1,\cdots,\psi^n)={T_0(x)}_p & \text{if $k=p$}
\end{cases}
.
$$
Therefore by \ref{rem:adjointnessofF[n]}, 
we  have the canonical morphism
$$\alpha^p:\cF[n]({T_0(x)}_p)(-p) \to ((F_px/F_{p-1}x)(p))(-p)=F_px/F_{p-1}x.$$
One can easily check that the morphism is an epimorphism. 

\sn
$\mathrm{(5)}$ 
By $\mathrm{(1)}$, 
we have the equalities 
$$
{F_px/F_{p-1}x}\isoto{T_0(\cF[n]({T_0(x)}_p)(-p))}_k \isoto
\begin{cases}
0 & \text{if $k\neq p $}\\
x_p/\im(\psi^1,\cdots,\psi^n) & \text{if $k=p$}
\end{cases}
$$
and ${T_0(\alpha^p)}_p=\id$. 
Hence we get the assertion. 

\sn
$\mathrm{(6)}$ 
Let $K^p$ be the kernel of $\alpha^p$, 
we have short exact sequences 
$$K^p \rinf \cF[n]({T_0(x)}_p)(-p) \rdef F_px/F_{p-1}x,$$
$$F_{p-1}x \rinf F_p x \rdef F_px/F_{p-1}x.$$
We call the long exact sequences of Koszul homologies 
associated with short sequences above $\mathrm{(I)}$, 
$\mathrm{(II)}$ respectively. 
By $\mathrm{(I)}$ 
and the assertions 
$\mathrm{(1)}$ and $\mathrm{(5)}$, 
we have the isomorphism 
$$T_1(F_px/F_{p-1}x)\isoto T_0(K^p).$$ 
We claim that the following assertion.

\begin{claim}
$T_1(F_px/F_{p-1}x)=0$ and $T_1(F_px)=0$.
\end{claim}

\sn
We prove the claim by descending induction of $p$. 
For sufficiently large $p$, we have $T_1(F_px)=T_1(x)$ 
and therefore it is trivial by the assumption. 
Then by $\mathrm{(II)}$ and $\mathrm{(3)}$, 
we have 
$$T_0(K^p)=T_1(F_px/F_{p-1}x)=0.$$ 
Therefore by $\mathrm{(2)}$, 
we have $K^p=0$. 
By $\mathrm{(I)}$ and $\mathrm{(1)}$, 
we have isomorphisms 
$$0=T_2(\cF[n]({T_0(x)}_p)(-p))
\isoto T_2(F_px/F_{p-1}x).$$ 
By $\mathrm{(II)}$, 
we get $T_1(F_{p-1}x)=0$. 
Hence we prove the claim and by $\mathrm{(2)}$, 
we get the desired result. 
\end{proof}

\begin{thm}
\label{cor:canisomoftf}
We have the canonical isomorphism
$$\bbZ[\sigma]\otimes_{\bbZ}K(\cA)\isoto K(\cA'_{\gr}[n])$$
which makes the diagram below commutative for any 
natural number $k$:
$$\footnotesize{
\xymatrix{
K(\cA) \ar[d]_{\sigma^k} \ar[rd]^{K(\cF[n](-k))}\\
\bbZ[\sigma]\otimes_{\bbZ}K(\cA) \ar[r]_{\sim} & K(\cA'_{\gr}[n]) & .
}}$$
\end{thm}

\begin{proof}[\bf Proof]
The inclusion functor 
$\cA'_{\gr,\tf}[n] \rinc \cA'_{\gr}[n]$ induce the isomorphism 
on their $K$-theories by 
\ref{lem:fundamental facts about} $\mathrm{(2)}$, 
\ref{prop:fun pro Kos hom} $\mathrm{(1)}$ and 
Corollary~3 of the resolution theorem in \cite{Qui73}. 
For each $m$, 
there are exact functors 
$$a:\cA'_{\gr,\tf,m}[n] \to \cA^{\times m+1},\ 
x \mapsto (T_0(x)_k)_{0\leq k\leq m},$$ 
$$b:\cA^{\times m+1} \to \cA'_{\gr,\tf,m}[n],\ 
{(x_k)}_{0\leq k\leq m} \mapsto \bigoplus_{k=0}^{m} \cF[n](x_k)(-k) .$$
Obviously $ab$ induces the identity map on their $K$-theories. 
On the other hand, 
any $x$ in $\cA'_{\gr,\tf,m}[n]$ has 
an exact characteristic filtration $F_{\bullet}x$ 
with $F_px/F_{p-1}x\isoto \cF[n]({T_0(x)}_p)(-p)$ 
by \ref{prop:fun pro Kos hom} $\mathrm{(6)}$, 
so applying Corollary~2 of 
the additivity theorem in \cite{Qui73}, 
we learn that $ba$ also induces the identity map 
on their $K$-theories. 
Therefore we have 
the isomorphism 
$$K(\cA'_{\gr,\tf,m}[n])\isoto 
\displaystyle{\bigoplus_{i=0}^m}K(\cA)\sigma^i.$$ 
Finally by taking the inductive limit, we get the desired isomorphism. 
\end{proof}

\section{Main theorem}
\label{sec:Main thm}

In this section, 
let us fix an essentially small noetherian abelian category $\cA$ which 
has enough projective objects. 
There is an exact functor $-\otimes_{\cA}\bbZ[t]$ from $\cA$ to $\cA[t]$, $a \mapsto a[t]$. 
The purpose of this section is to study the induced map 
from $-\otimes_{\cA}\bbZ[t]$ on $K$-theory. 
More precisely, we will prove the following theorem.

\begin{thm}
\label{thm:main thm}
The functor $-\otimes_{\cA} \bbZ[t]:\cA \to \cA[t]$ induces the 
isomorphism on their $K$-theories
$$K(\cA)\isoto K(\cA[t]).$$
\end{thm}

\subsection{Nilpotent objects in $\cA'_{\gr}[2]$}
\label{subsec:Nilp}

In this subsection, 
we will define the category $\cA'_{\gr,\nil}[2]$ 
of nilpotent objects 
in $\cA'_{\gr}[2]$. 
We also study the relationship $\cA'_{\gr}[2]$ with 
$\cA[t]$ and calculate the $K$-theory of $\cA'_{\gr,\nil}[2]$. 
For simplicity in this subsection we write $\psi=\psi^1$ and $\phi=\psi^2$ 
and for any object $x$ in $\cA$, we write $\cF[2](x)$ by $x[\psi,\phi]$.

\begin{df}
\label{df:nilpotent obj}
An object $x$ in $\cA'_{\gr}[2]$ is ({\bf $\psi$-}) {\bf nilpotent} 
if there is an integer $n$ such that 
$$\psi^{n,k}_x=0$$ 
for any non-negative integer $k$.
We write the full subcategory of 
nilpotent objects in $\cA'_{\gr}[2]$ by 
$\cA'_{\gr,\nil}[2]$. 
\end{df}

\begin{lem}
\label{lem:niliSerresubcat}
The category $\cA'_{\gr,\nil}[2]$ 
is a Serre subcategory of $\cA'_{\gr}[2]$. 
In particular $\cA'_{\gr,\nil}[2]$ 
is an abelian category.
\end{lem}

\begin{proof}[\bf Proof]
The assertion that $\cA'_{\gr,\nil}[2]$ 
is closed under 
sub and quotient objects and finite direct sum is easily proved. 
Now we intend to prove the following assertion. 
For a short exact sequence 
$x \rinf y \rdef z$ in $\cA'_{\gr}$, 
let $i$, $j$ be integers such that 
$\psi_x^i=0$ and $\psi^j_z=0$. 
Then we can easily prove that $\psi^{i+j}_y=0$. 
Therefore $\cA'_{\gr,\nil}[2]$ 
is closed under extensions in $\cA'_{\gr}[2]$. 
\end{proof}

\begin{prop}
\label{prop:cannonical isom}
If $\cA$ has enough projective objects, 
then there is a canonical isomorphism 
$$\cA'_{\gr}[2]/\cA'_{\gr,\nil}[2]\isoto \cA[t].$$
\end{prop}

\begin{proof}[\bf Proof]
We define the functor 
$$\Theta:\cA'_{\gr}[2] \to \End\Lex \cA,\ \  
x \mapsto 
(\coker(\overset{\infty}{\underset{n=0}{\bigoplus}}x_n 
\onto{\id-\psi_x}
\overset{\infty}{\underset{n=0}{\bigoplus}}x_n), \phi_x)$$ 
where 
$\displaystyle{\psi_x=\overset{\infty}{\underset{n=0}{\bigoplus}} \psi_n}$ 
and 
$\displaystyle{\phi_x=\overset{\infty}{\underset{n=0}{\bigoplus}} \phi_n}$. 
For any object $x$ in $\cA'_{\gr,\nil}[2]$, 
assume that $\psi^{m,k}_x=0$ for any non-negative integer $k$. 
Then $\displaystyle{\sum_{i=0}^{m-1}\psi_x^i}$ 
is the inverse map of $\id-\psi_x$. 
Therefore $\Theta(x)=0$. 
Next we prove that $\id-\psi_x:\bigoplus x_n \to \bigoplus x_n$ 
is a monomorphism. 
Let $K$ be the kernel of $\id-\psi_x$. 
Assume that $K$ is not the zero object. 
Then there is the maximum integer $m$ such that 
$\displaystyle{K\cap\bigoplus_{n=0}^{m}x_n=0}$. 
Then we have equalities
$$0=(\id-\psi_x)(K\cap x_{m+1})\cap x_{m+1} =K\cap x_{m+1}.$$ 
It contradicts the maximality of $m$. 
Therefore for any conflation $x \rinf y \rdef z$ in $\cA'_{\gr}[2]$, 
by applying the snake lemma to the commutative diagram below, 
we notice that $\Theta$ is an exact functor. 
$$\footnotesize{\xymatrix{
\bigoplus x_i \ar[d]_{\psi_x} \ar@{>->}[r] & 
\bigoplus y_i \ar[d]_{\psi_y} \ar@{->>}[r] & 
\bigoplus z_i \ar[d]^{\psi_z}\\
\bigoplus x_i \ar@{>->}[r] & 
\bigoplus y_i \ar@{->>}[r] & 
\bigoplus z_i  &.
}}$$
Moreover for any object $x$ in $\cA$ and any positive integer $k$, 
$\Theta(x[\psi,\phi](-k))=x[t]$. 
Therefore by 
\ref{rem:other df of S-poly cat} and 
\ref{lem:fundamental facts about} $\mathrm{(2)}$, 
we notice that 
$\Theta$ induces the exact functor 
$$\Theta':\cA'_{\gr}[2]/\cA'_{\gr,\nil}[2] \to \cA[t].$$ 
Since $\cA$ has enough projective objects, 
for any object $x$ in $\cA[t]$, 
there is a finite presentation 
$$a[t] \onto{\sum_{i=0}^m c_it^i} b[t] \to x \to 0$$ 
where $a$ and $b$ are projective objects in $\cA$ 
and therefore
$a[t]$ and $b[t]$ are also projective objects in $\cA[t]$ 
by \ref{lem:projectivity}. 
(For the notation $\sum_{i=0}^m c_it^i$, 
see \ref{ex:morphismofa[t]}.)
Then we define the functor 
$$\Delta(x):=\coker(a[\psi,\phi](-m)
\onto{\sum_{s+t=m}c_s\psi^s\phi^t} b[\psi,\phi]).$$ 
(For the notation $\sum_{s+t=m}c_s\psi^s\phi^t$, 
see \ref{ex:mor between freeobj}.) 
Since $a[\psi,\phi]$ and $b[\psi,\phi]$ are projective in $\cA'_{\gr}[2]$ 
by \ref{lem:fundamental facts about} $\mathrm{(3)}$, 
the association $\Delta$ is 
well-defined and it gives the inverse functor of $\Theta'$.
\end{proof}

\begin{cor}
\label{cor:canisom} 
We have a fibration sequence 
$$K(\cA'_{\gr,\nil}[2]) \to K(\cA'_{\gr}[2]) \to K(\cA[t]).$$
\end{cor}

\begin{prop}
\label{prop;devissage}
The inclusion functor $\cA'_{\gr}[1] \rinc \cA'_{\gr,\nil}[2]$ 
defined by $(x,\psi^1) \mapsto (x,\psi^1,\phi=0)$ induces 
the isomorphism on their $K$-theories.
\end{prop}

\begin{proof}[\bf Proof]
First notice that 
$\cA'_{\gr}[1]$ is closed under admissible sub and quotient objects 
in $\cA'_{\gr,\nil}[2]$. 
Moreover for any $x$ in $\cA'_{\gr,\nil}[2]$, 
let us consider the filtration $\{\im\psi^k\}_{k\in\bbN}$ of $x$. 
Then for each $k$, $\im\psi^k/\im\psi^{k+1}$ is 
isomorphic to an object in $\cA'_{\gr}[1]$. 
Therefore we get the desired result by the d\'evissage theorem.
\end{proof}

\begin{cor}
\label{cor:devissage}
We have the canonical isomorphism 
$$K(\cA)\otimes_{\bbZ}\bbZ[\sigma] \isoto K(\cA'_{\gr,\nil}[2]).$$
\end{cor}

\subsection{The proof of main theorem}
\label{subsec:The proof of main theorem}

In this subsection, 
we will finish the proof of the main theorem. 
The key lemma is the following.

\begin{lem}
\label{lem:commutative diagram}
There is the commutative diagram below
$$\xymatrix{
\bbZ[\sigma]\otimes_{\bbZ}K(\cA) \ar[r] \ar[d]_{(1-\sigma)\otimes \id} & K(\cA'_{\gr,\nil}[2]) 
\ar[d]\\
\bbZ[\sigma]\otimes_{\bbZ}K(\cA) \ar[r] & K(\cA'_{\gr}[2]) & .
}$$
\end{lem}

\begin{proof}[\bf Proof]
An object $a$ in $\cA$ goes to $(a[\psi],\psi,0)$ by 
the functors $\cA \to \cA'_{\gr,\nil}[2] \to \cA'_{\gr}[2]$ and 
goes to $a[\psi,\phi]$ 
by the functor $\cF[2]:\cA \to \cA'_{\gr}[2]$. 
Moreover 
let us notice that the functor $\cF[2](-k):\cA\to\cA'_{\gr}[2]$ induces 
$\bbZ[\sigma]\otimes_{\bbZ}K(\cA) \onto{\sigma^k} 
\bbZ[\sigma]\otimes_{\bbZ}K(\cA)\isoto K(\cA'_{\gr}[2])$ 
by \ref{cor:canisomoftf}. 
On the other hand, 
for each object $a$ in $\cA$, 
there is an exact sequence in $\cA'_{\gr}[2]$ 
$$a[\psi,\phi](-1) \overset{\phi}{\rinf} a[\psi,\phi] \rdef (a[\psi],\psi,0).$$
By the additivity theorem, 
this implies that the diagram in the statement 
is commutative. 
\end{proof}

\begin{proof}[\bf Proof of \ref{thm:main thm}]
The assertion follow from the commutative diagram of exact sequences below. 
$$\xymatrix{
\bbZ[\sigma]\otimes_{\bbZ} K(\cA) \ar@{>->}[r]^{(1-\sigma)\otimes \id} \ar[d]_{\wr} & \bbZ[\sigma]\otimes_{\bbZ} K(\cA) \ar@{->>}[r] \ar[d]^{\wr} & 
K(\cA) \ar[d]^{K(-\otimes_{\cA}\bbZ[t])}\\
K(\cA'_{\gr,\nil}) \ar[r] & K(\cA'_{\gr}) \ar[r] & K(\cA[t]) & .
}$$
\end{proof}






\end{document}